\documentclass[12pt]{amsart}
\usepackage{geometry}
\usepackage{graphicx, amsfonts, amssymb}
\usepackage{mathrsfs, amsmath, amsthm}
\usepackage{verbatim}
\usepackage{relsize}
\usepackage{epsfig}
\usepackage{xcolor}

\newcommand{\hol}{\rm{Hol}}
\DeclareMathOperator{\og}{O}

\newcommand{\defeq}{\stackrel{\text{def}}{=}}
\def\D{{\mathbb D}}

\def\({\left(}       \def\){\right)}

\newtheorem{theorem}{Theorem}

\newtheorem{corollary}{Corollary}
\newtheorem{proposition}{Proposition}
\theoremstyle{definition}

\numberwithin{equation}{section}
\theoremstyle{theorem}
\newtheorem{other}{\bf Theorem}              

\newenvironment{Pf}{\noindent{\emph{Proof of}}}{$\Box$ }
 \DeclareMathOperator{\ogr}{O}

\begin{document}
\title[Rhaly and Generalized Hilbert operators ]
{Rhaly and Generalized Hilbert operators between spaces of analytic functions }

\author[P.~Galanopoulos]{Petros Galanopoulos}
 \address{Department of Mathematics,
Aristotle University of Thessaloniki, 54124, Thessaloniki, Greece}
\email{petrosgala@math.auth.gr}
\author[D.~Girela]{Daniel Girela}
 \address{An\'alisis Matem\'atico,
Universidad de M\'alaga, Campus de Teatinos, 29071 M\'alaga, Spain}
\email{girela@uma.es}

\subjclass[2020]{Primary 30H10; 47B91; 42B30}
\keywords{Hardy spaces, Bergman spaces, The
Ces\`{a}ro operator, Rhaly operators, Generalized Hilbert operators, Mean Lipschitz spaces}
\begin{abstract}
In this article we address 
the question of characterizing the boundedness of Rhaly and generalized Hilbert operators when acting between distinct
spaces of analytic functions in the unit disc
$\mathbb D$.
 Let  $(\eta )=\{ \eta_n\}_{n=0}^\infty $ be a sequence of complex numbers such that the $g(z)=g_{(\eta)}(z)=\sum_{n=0}^{\infty} \eta_n z^n$ is  analytic in $\mathbb D$. Assume that $f(z)=\sum_{k=0}^{\infty} a_k z^k$ belongs to the Hardy space $ H^1 $ then the associated Rhaly operator
is defined 
as
$$
\mathcal R_{(\eta )}(f)(z)=\sum_{n=0}^{\infty} \eta_n \sum_{k=0}^n a_k\, z^n\,. \quad z\in \mathbb D\,.
$$
 In addition, 
 the corresponding  generalized Hilbert operator is defined 
 as
$$
\mathcal H_{g}(f)(z)=\sum_{n=0}^{\infty} (n+1) \eta_{n+1} \sum_{k=0}^{\infty} \frac{a_k}{n+k+1}\, z^n=\int_0^1 f(t) g'(tz)\,dt\,, \quad z\in \mathbb D\,.
$$
We prove that the boundedness of $\mathcal R_{(\eta )} : H^p \to H^q$  is equivalent to that of $\mathcal H_{g} : H^p \to H^q$, when $1<p<q\leq 2$ or when $1<p\leq 2 <q$, and it is characterized exactly by the membership of the function $g$
 in the mean Lipschitz space $\Lambda ^q_{1/p}$.
 The case $1<p=q \leq 2$ was obtained previously by the authors.
 Our result confirms in a strong way the bond between these two operators.
 \par On the other hand, this condition is no longer true when the two operators act between distinct Bergman spaces.
  Specifically, if $2<p<q <\infty$ then
  we prove that  $\mathcal R_{(\eta)} : A^p \to A^q$ boundedly if and only if $g\in \Lambda^q_{2/p-1/p} $. Pelaez and Rattya prove that the same holds for the $\mathcal H_{g}$. The $2<p<q<\infty$ 
   is the natural range of values for the $\mathcal H_{g}$ to be well defined on $A^p$. This difference on the action of the two operators between  Hardy and  Bergman spaces is surprising compared to the case  $p=q$ where the $g\in \Lambda ^p_{1/p}$ is the exact condition for both operators to be bounded either in  Hardy or in  Bergman spaces.
\par Furthermore, we deal with the action $\mathcal R_{\eta} : D^p_{\alpha} \to D^p_{\beta}$ of the Rhaly operator between weighted Dirichlet spaces extending previous work of O.Blasco with the authors for the case $p=2$. 
Finally, we answer completely the question about the action of $\mathcal H_{g}$ between spaces of analytic functions whose derivative is in a Hardy space. To be precise, we prove that $\mathcal H_{g} : S^p \to S^q$ is bounded if and only it is compact which in its turn is equivalent to  $g\in S^q$  when $ 1\leq p<\infty,\, 1<q<\infty$. According to a previous result of the authors the latter is the case for the Rhaly operator too.


\end{abstract}

 \thanks{This research has been supported in part by a grant from \lq\lq El Ministerio de
Ciencia e Innovaci\'{o}n\rq\rq , Spain Project PID2022-133619NB-I00
and by grant from la Junta de Andaluc\'{\i}a
FQM-210-G-FEDER)}\maketitle

\section{Introduction and main results}\label{intro}
\par\smallskip Let
$\mathcal S$ be the space of all complex sequences $(a)=\{ a_k\}
_{k=0}^\infty \subset \mathbb C$. If $(\eta )\,=\,\{
\eta_n\}_{n=0}^\infty \in \mathcal S$, the Rhaly operator $\mathcal
R_{(\eta )}\,=\,\mathcal R_{\{ \eta _n\} }
:\mathcal
S\,\rightarrow\,\mathcal S$ is defined as follows: If $\{
a_k\}_{k=0}^\infty \in \mathcal S$, then
$$\mathcal R_{(\eta )}(\{ a_k\}_{k=0}^\infty )\,=\,\left \{
\eta_n\sum_{k=0}^na_k\right \}_{n=0}^\infty .$$ Taking $(\eta )=\{
\frac{1}{n+1} \} $ the operator $\mathcal R_{(\eta )}$ reduces to
the classical Ces\`{a}ro operator $\mathcal C$. Thus Rhaly operators
are a natural generalization of the Ces\`{a}ro operator which is
known to be bounded on $\ell^p$ for $1<p<\infty $ \cite{H-20,HLP,L}.
The sequences $(\eta)=\{ \eta_n\}_{n=0}^\infty $ for which the operator
$\mathcal R_{(\eta) }$ is either bounded or compact on the
space $\ell^p$ for any $p\in (1, \infty )$ have been characterized
in \cite{GGPr}. A characterization of the sequences $(\eta )=\{ \eta_n\}$
for which $\mathcal R_{(\eta) }$ belongs to the Schatten class
$\mathcal S^q(\ell^2)$ ($1<q<\infty $) has been obtained in
\cite{BDS}.
\par\medskip
Particular cases of Rhaly operators of special interest are those
where the sequence $(\eta )=\{ \eta _n\} $ is the sequence of moments of a
complex Borel measure $\mu $ on $\mathbb D$: If $\mu $ is a complex
Borel measure in $\mathbb D$ and, for $n=0, 1, 2, \dots $, $\mu_n$
denotes the moment of order $n$ of $\mu $, that is,
$\mu_n\,=\,\int_{\mathbb D}w^n\,d\mu (w)$,\,($n=0, 1, 2, \dots $),
the operator $\mathcal R_{\{\mu _n\} }$ will be denoted by $\mathcal
C_{\mu }$. When $\mu $ is the Lebesgue measure on the radius $[0,
1)$ the operator $\mathcal C_{\mu }$ is the Ces\`{a}ro operator
$\mathcal C$.
\par \medskip The generalized Hilbert operator  introduced by the sequence $(\eta )\,=\,\{
\eta_n\}_{n=0}^\infty \in \mathcal S$ 
 on the space $ \mathcal S$ is
the formally defined operator
$$
\{ a_k\}_{k=0}^\infty \longrightarrow  \left \{
(n+1) \eta_{n+1} \sum_{k=0}^{\infty} \frac{a_k}{n+k+1} \right \}_{n=0}^\infty .
$$
 In order to be consistent to the notation commonly used in the literature for the generalized Hilbert operators
 we will denote
 the sequence that introduce the operator as $g=\,\{
\eta_n\}_{n=0}^\infty$ and therefore
the associated generalized Hilbert operator on $ \mathcal S$  as
$$
\mathcal H_{g}(\{ a_k\}_{k=0}^\infty)=\left \{
(n+1) \eta_{n+1} \sum_{k=0}^{\infty} \frac{a_k}{n+k+1} \right \}_{n=0}^\infty\,.
$$
Under particular conditions on the $\{ a_k\}
_{k=0}^\infty$ the $\mathcal H_g$ becomes well defined. For example, this is true when $\{ a_k\}
_{k=0}^\infty \in l^p \,(1< p<\infty)$.  If we choose the sequence $g=\{\frac{1}{n} \}_{n=1}^{\infty}$
 then the operator $\mathcal H_{g}$ reduces to
the Hilbert operator. This is the operator introduced by the most important representative of the Hankel matrices that is the Hilbert matrix
 $\mathcal H=[c_{n,k}]_{n,k \geq 0}=[\frac{1}{n+k+1}]_{n,k \geq 0}$ by multiplication of $\mathcal H$ with the column matrix with entries the terms of the sequence $\{a_k\}_{k=0}^\infty$. Thus, from this point of view, the generalized Hilbert  operators
are a natural extension of the Hilbert operator which is
known to be bounded on $\ell^p,\,1<p<\infty $ \cite{HLP}\,. 
\par Although the formulas of definition of $\mathcal R_{(\eta)}$ and $\mathcal H_g$ are different it seems that there is a strong connection between them. This can be realized initially through their action on the sequence spaces. Despite the fact that in \cite{GGPr}  it is not written down a concrete proof for the boundedness of $\mathcal H_{g}$ on $l^p$ spaces, looking  carefully at the proof presented there about $\mathcal R_{(\eta )}$ we realize that in fact it is established the following.
\begin{other}\label{RHglp}
Let $1<p<\infty$ and $g=(\eta)=\,\{\eta_n\}_{n=0}^\infty \in \mathcal S$. The $\mathcal R_{(\eta )} : l^p \to l^p$ is bounded if and only if the $\mathcal H_{g} :l^p \to l^p$ is bounded.
 \end{other}
\par\medskip The next natural step in the study of the Rhaly and the generalized Hilbert operators  is to consider them on analytic functions.  Let $\mathbb D=\{ z\in \mathbb C: \vert z\vert <1\} $
be the unit disc in the complex plane and let $\hol (\mathbb D)$ be
the space of analytic functions in $\mathbb D$.
If $(\eta )=\{ \eta _n\}_{n=0}^\infty \in \mathcal S$
then we can view the $\mathcal
R_{(\eta )}$ as an operator acting on analytic
functions in the disc: If $f\in \hol (\mathbb D)$, $f(z)=\sum
_{k=0}^\infty a_k z^k $ ($z\in \mathbb D$), formally, we set
$$\mathcal R_{(\eta )}(f)(z)\,=\,
\sum_{n=0}^\infty \eta _n\left
(\sum_{k=0}^na_k\right )z^n,\quad z\in\mathbb D,$$ whenever the
right hand side makes sense and defines an analytic function in
$\mathbb D$. Notice that the Ces\`{a}ro operator
$$\mathcal C(f)(z)=\sum_{n=0}^\infty \left
(\frac{1}{n+1}\sum_{k=0}^na_k\right )z^n
$$
 and the more general operators
$$\mathcal C_\mu
(f)(z)=\sum_{n=0}^\infty \left (\mu _n\sum_{k=0}^na_k\right
)z^n
$$
become linear operators from $\hol (\mathbb D)$ into itself.
 We recall that the Ces\`{a}ro operator is bounded on the Hardy spaces $H^p$ ($0<p<\infty $) and on the Bergman spaces $A^p $ ($0<p<\infty $). 
 This has been proved
over the years using different methods \cite{Andersen, Mi, Nowak,
    Sis-1987, Sis-1990, Sis-1996, Stem}.
 The boundedness of $\mathcal C_\mu $ on spaces of analytic functions in
$\mathbb D$ has been studied by several authors (see, e.\,\@g.,
\cite{BaoSWu, BBJ, Blasco-2024, Bl-Ma, GGMM, JT, T}).
\par\medskip The sequence $(\eta )=\{ \eta _n\}_{n=0}^\infty$ may not be
bounded. Hence, contrary to what happens with the operators
$\mathcal C$ and  $\mathcal C_\mu $, it is not obvious that
$\mathcal R_{(\eta )}$ is  well defined  from $\hol (\mathbb D)$
into itself. Proposition~1 of \cite{Bl-Ga-Gi} asserts that if $X$ is
a subspace of $\hol (\mathbb D)$ which contains the constants, then
$\mathcal R_{(\eta )}$ is a well defined linear operator from $X$
into $\hol (\mathbb D)$ if and only if the power series
$\sum_{n=0}^\infty \eta _nz^n$ defines an analytic function in
$\mathbb D$. From now on, if the power series $\sum_{n=0}^\infty
\eta _nz^n$ defines an analytic function in $\mathbb D$ we will
denote it as
$$g(z)=
\,\sum_{n=0}^\infty \eta
_nz^n,\quad z\in \mathbb D.$$
\par In a similar way we can view $\mathcal H_{g}$
as an operator acting on analytic
functions in the disc: if $g(z)=\,\sum_{n=0}^\infty \eta
_nz^n \in \hol (\mathbb D)$ then we  formally set
$$\mathcal H_{g}(f)(z)\,
=\sum_{n=0}^\infty (n+1)\eta _{n+1}\left
(\sum_{k=0}^{\infty} \frac{a_k}{n+k+1}\right )z^n,\quad z\in\mathbb D,$$
 whenever the
right hand side makes sense and defines an analytic function in
$\mathbb D$. As it is established in \cite{GGPSi} this is true when $f(z)=\sum
_{k=0}^\infty a_k z^k $  runs in the Hardy space $H^1$ or in the Bergman spaces $ A^p$ when $p>2$.
  Furthermore,  in \cite{GGPSi} it is proved that
  \begin{equation}\label{series=integral}
 \mathcal H_{g}(f)(z)=\int_0^1 f(t)\,g'(tz)\,dt \,,\,\,\,z\in \mathbb D\,.
\end{equation}
where the integral stated above is absolutely convergent when $f\in H^1$ or $f\in A^p (p>2)$.

As in the case of the Rhaly operators, the choice  $g(z)=\,\sum_{n=0}^\infty \frac{1}{n+1} z^n $ leads to a known operator. Here is the Hilbert. 
  In particular, relation (\ref{series=integral})  can be written as
$$
\mathcal H(f)(z)\,
=\sum_{n=0}^\infty \left
(\sum_{k=0}^{\infty} \frac{a_k}{n+k+1}\right )z^n = \int_0^1 \frac{f(t)}{1-tz} dt\, ,\quad z\in\mathbb D\,,
$$
whenever $f\in H^1$ or $f\in A^p (p>2)$.
The term Hilbert operator is justified by the fact that coefficients of the series is the outcome of the multiplication of the Hilbert matrix with the one column matrix with entries the Taylor coefficients of $f$.
 The boundedness of $\mathcal H$ on Hardy and Bergman spaces was established in \cite{DiSis}, \cite{Di}.
At this point it is worth of mentioning that a broader  version of the Hilbert  operator acting on $\hol(\mathbb  D)$ has attracted the attention of the researchers as well. This  corresponds to the operators introduced by the
Hankel matrices $ \mathcal H_{\mu}=[\mu_{n+k}]_{n,k\geq 0}$ where $\mu_n=\int_0^1 t^n d\mu(t)$ are the moments of a  positive Borel measure $\mu $ on $[0,1)$. If $\mu $ is the Lebesgue measure then we get the Hilbert matrix.
About the action of $\mathcal H_{\mu}$ on the Hardy and the Bergman spaces see \cite{AGG}, \cite{BaoS}, \cite{GP}, \cite{GPCh}.
\par Taking into account all the aforementioned results about the action of  particular versions of the  Rhaly and generalized Hilbert operators on the Hardy and Bergman spaces and Theorem \ref{RHglp}
 we are allowed to expect that their boundedness and compactness on these spaces should be characterized by the same condition on the function $g(z)=\sum_{n = 0}^{\infty} \eta_n z^n \in \hol(\mathbb D)$. This is actually the case.
In view of the identification of $H^2$ and $\ell^2$, the sequences $(\eta )$ for which $\mathcal R_{(\eta )}$ is bounded
on $H^2$ are characterized in \cite{GGPr}.
In \cite{GG-2026JGA},\cite{Bl-Ga-Gi2} it is provided the exact condition in order the $\mathcal R_{\eta}$ to be bounded (compact) on $H^p (1<p\leq 2), A^p (1<p<\infty).$ The answer about $\mathcal H_g$
on $H^p (1<p\leq 2)$ and $A^p (2<p<\infty)$ was given in \cite{GGPSi}.
 We will present these results  in  a unified way. For simplicity we will adopt the following notation. If X and Y are two Banach spaces, $\mathcal B(X, Y )$ will be the space of continuous linear operators from X into Y. Also $\mathcal B(X)$ will stand for $\mathcal B(X, X)$. The compact operators will be denoted as $\mathcal K(X,Y)$, $\mathcal K(X)$.
 In general, we say that $g\in \Lambda^p_{\alpha}$, $1<p< \infty,\, 0<\alpha <1$ if and only if
 $$
 M_p(r, g^\prime )\,=\,\og \left
 (\frac{1}{(1-r)^{1-\alpha}}\right ),\quad
 \text{ as $r\to 1$,}
 $$
 If a ``little oh" condition holds then we say that $g\in\lambda^p_{\alpha}\,.$
  Summing up, we state the following theorems.
  \begin{other}\label{summing up}
    Let $g(z)=\sum_{n=0}^\infty \eta_n z^n \in \hol(\mathbb D)$. If  $1<p\leq 2$ then the following are equivalent\\
    $(i)$ $\mathcal R_{\eta} \in \mathcal B(H^p) \,\,( \mathcal K(H^p))$\\
    $(ii) $ $\mathcal H_{g}\in \mathcal B(H^p)\,\, ( \mathcal K(H^p))$\\
    $(iii)$ $g\in \Lambda^p_{\frac 1p}\,\, (\lambda^p_{\frac 1p})$
    \end{other}
\begin{other}\label{summing up2}
    Let $g(z)=\sum_{n=0}^\infty \eta_n z^n \in \hol(\mathbb D)$.
     Assume that $1<p <\infty$ then the following are equivalent\\
    $(i)$ $\mathcal R_{\eta} \in \mathcal B(A^p) \,\,( \mathcal K(A^p))$\\
    $(ii)$ $g\in \Lambda^p_{\frac 1p}\,\, (\lambda^p_{\frac 1p})$\\
\end{other}
\begin{other}\label{summing up3}
    Let $g(z)=\sum_{n=0}^\infty \eta_n z^n \in \hol(\mathbb D)$. If $2<p<\infty$ then the following are equivalent\\
        $(i)$ $\mathcal H_{g} \in \mathcal B(A^p) \,\,( \mathcal K(A^p))$\\
    $(ii)$ $g\in \Lambda^p_{\frac 1p}\,\, (\lambda^p_{\frac 1p})$\\
\end{other}
 \par Here we go one step further.
We confirm 
 the strong bond of the two operators by studying their action between distinct Hardy, Bergman spaces. In \cite{Bl-Ma} the authors considered this problem for $\mathcal C_{\mu}$ in the context of
 mixed normed spaces. However, even for this particular version of Rhaly operators, they didn't confront the  case of the Hardy spaces. We prove the following.
\begin{theorem}\label{Rhaly Hardy}
Let $g(z)=\sum_{n=0}^\infty \eta_n z^n \in \hol(\mathbb D)$.\\
$(i)$ If\, $1<p<q\leq 2 $ or $1<p\leq 2< q <\infty $ \, and
   $g\in \Lambda^q_{\frac 1p}$ then $\mathcal R_{\eta}\in \mathcal B(H^p,H^q)$.\\
$(ii)$ If \, $1\leq p <\infty, \,\, 1<q<\infty $, \, and \, $\mathcal R_{\eta}\in \mathcal B(H^p,H^q)$ then $g\in \Lambda^q_{\frac 1p}$\,.\\
$(iii)$ If \, $1<p <q\leq 2 $ or $1<p\leq 2< q <\infty $ \, and \, $g\in \lambda^q_{\frac 1p}$ then \,\,$\mathcal R_{\eta}\in \mathcal K(H^p,H^q)$\\
$(iv)$ If \, $1\leq p <\infty, \,\, 1<q<\infty $, \, and \,$\mathcal R_{\eta}\in \mathcal K(H^p,H^q)$ then $g\in \lambda^q_{\frac 1p}$.\\
\end{theorem}

\par Looking at the action of the generalized Hilbert operator we conclude that it  behaves the same way as the Rhaly.
\begin{theorem}\label{Hg Hardy}
    Let $g(z)=\sum_{n=0}^\infty \eta_n z^n \in \hol(\mathbb D)$.\\
    $(i)$ If \, $1<p<q\leq 2 $ or $1<p\leq 2 <q <\infty $ \, and \, $g\in \Lambda^q_{\frac 1p}$ then $\mathcal H_g\in \mathcal B(H^p,H^q)$\,.\\
    $(ii)$ If \, $1<p < q < \infty \,\,$ \, and \, $\mathcal H_g\in \mathcal B(H^p,H^q)$ then $g\in \Lambda^q_{\frac 1p}$\,.
    \\
    $(iii)$ If \, $1<p<q\leq 2 $ or $1<p\leq 2 <q <\infty$ \, and \, $g\in \lambda^q_{\frac 1p}$ then $\mathcal H_g\in \mathcal B(H^p,H^q)$\,.\\
    $(iv)$ If \, $1<p < q < \infty $ \, and \, $\mathcal H_g\in \mathcal K(H^p,H^q)$ then $g\in \lambda^q_{\frac 1p}$\,.
    \end{theorem}
A different case of the action of $\mathcal H_g$ between distinct Hardy spaces  has been considered recently in \cite{NPW}. Summarizing the two Theorems we state the following.
    \begin{corollary}\label{Hg}
            Let $g(z)=\sum_{n=0}^\infty \eta_n z^n \in \hol(\mathbb D)$and $1<p <q\leq 2 $ or $1<p\leq 2< q <\infty$.\\
    $(i)$ $\mathcal R_{\eta}\in \mathcal B(H^p,H^q)$ if and only if  \,\,$\mathcal H_{g}\in \mathcal B(H^p,H^q)$ if  and only if $g\in \Lambda^q_{\frac 1p}$.\\
    $(ii)$ $\mathcal R_{\eta}\in \mathcal K(H^p,H^q)$ if and only if  \,\,$\mathcal H_{g}\in \mathcal K(H^p,H^q)$ if and only if $g\in \lambda^q_{\frac 1p}$
\end{corollary}
\par Although the Bergman spaces are part of the mixed normed spaces, here we consider the general version of a Rhaly operator instead of $\mathcal C_{\mu}$ as in \cite{Bl-Ma}.
\begin{theorem}\label{Rhaly Bergman}
    Let $g(z)=\sum_{n=0}^\infty \eta_n z^n \in \hol(\mathbb D)$ and
    $1<p<q<\infty$ such that $\frac{2}{p} - \frac{1}{q} < 1$ then  \\
$(i)$\,\,$\mathcal R_{\eta}\in \mathcal B(A^p,A^q) \,\,\, \text{if and only if} \,\,\, g\in \Lambda^q_{2/p-1/p}\,.$\\
$(ii)$\,\,$\mathcal R_{\eta}\in \mathcal K(A^p,A^q) \,\,\, \text{if and only if} \,\,\, g\in \lambda^q_{2/p-1/p}\,.$
\end{theorem}
At this point we recall that, when $2<p <q<\infty $,\,\, $\frac{2}{p} - \frac{1}{q} < 1,$ then $\mathcal H_{g}\in \mathcal B(A^p,A^q)(\mathcal K(A^p,A^q))$  if and only if  $g\in \Lambda^q_{2/p-1/p}\,(\lambda^q_{2/p-1/p})$. See \cite{PR}.
Combining the latter with Theorem \ref{Rhaly Bergman} we get the following.
\begin{corollary}\label{RB}
    Let $g(z)=\sum_{n=0}^\infty \eta_n z^n \in \hol(\mathbb D)$ and
    $2<p<q<\infty$, \,$\frac{2}{p} - \frac{1}{q} < 1$.\\
    $(i)$ $\mathcal R_{\eta}\in \mathcal B(A^p,A^q)$ if and only if  \,\,$\mathcal H_{g}\in \mathcal B(A^p,A^q)$ if  and only if $g\in \Lambda^q_{2/p-1/p}$.\\
    $(ii)$ $\mathcal R_{\eta}\in \mathcal K(A^p,A^q)$ if and only if  \,\,$\mathcal H_{g}\in \mathcal K(A^p,A^q)$ if and only if $g\in \lambda^q_{2/p-1/p}$.
\end{corollary}
Notice that the condition on the symbol $g\in \hol(\mathbb D)$ in the case of  the Hardy spaces is different from that in the Bergman.
Thus, it breaks down the belief based on the results of the case $p=q$ (\cite{Bl-Ga-Gi}, \cite{Bl-Ga-Gi2}, \cite{GGPSi})
that the two operators behave the same either in the Hardy or in the Bergman spaces.
\par\medskip In \cite{Bl-Ga-Gi} the authors dealt with the problem $\mathcal R_{\eta} : D^2_{\alpha} \to D^2_{\beta}$ that is the action of the Rhaly operator between different Hilbert type Dirichlet spaces. We extend this study to the case of Banach type Dirichlet spaces as below.
 \begin{theorem}\label{Rhaly Dpa Dpb}
    Let  $g(z)=\sum_{n=0}^\infty \eta_n z^n \in \hol(\mathbb D)$ . If $1<p<\infty, a,b>-1$, $ p-2<a<\min\{p-1+b,2p-2\}$ and $a-b+1>0$  then \\
    $(i)$ \,\, $\mathcal R_{\eta}\in \mathcal B(D^p_a,D^p_b)$ if and only if $g\in \Lambda^p_s$ where $s=\frac{a-b+1}{p}$ \,.\\
    $(ii)$ \,\, $\mathcal R_{\eta}\in \mathcal K(D^p_a,D^p_b)$ if and only if $g\in \lambda^p_s$ where $s=\frac{a-b+1}{p}$ \,.
\end{theorem}
\par Finally, we answer completely the open question posed in \cite{T2} about the action of $\mathcal H_g $  on the spaces $\mathcal S^p$ consisted of analytic functions on the unit disc with first derivative in  $H^p$.
\begin{theorem}\label{Hg on Sp}
    Let $g(z)=\sum_{n=0}^\infty \eta_n z^n \in \hol(\mathbb D)$. If $1\leq p <\infty, 1< q <\infty$ then the following are equivalent\\
$(i)\,\,\mathcal H_g \in \mathcal B(\mathcal S^p, \mathcal S^q)$ \\
 $(ii)\,\,\mathcal H_g  \in \mathcal K(\mathcal S^p, \mathcal S^q)$\\
 $(iii) \,\,g\in S^q$
 \end{theorem}
Combining Theorem \ref{Hg on Sp} with the results
about the Rhaly operator proved in \cite{GG-2025} we are allowed to state that
\begin{corollary}\label{RHgSp}
    Let $g(z)=\sum_{n=0}^\infty \eta_n z^n \in \hol(\mathbb D)$. If $1\leq p  <\infty, 1<q<\infty $ then the following are equivalent\\
    $(i)\,\,\mathcal H_g \in \mathcal B(\mathcal S^p, \mathcal S^q)$\\
    $(ii)\,\,\mathcal R_{\eta} \in \mathcal B(\mathcal S^p, \mathcal S^q)$\\
    $(iii)\,\,\mathcal H_g \in \mathcal K(\mathcal S^p, \mathcal S^q)$\\
    $(iv) \,\,\mathcal R_{\eta} \in \mathcal K(\mathcal S^p, \mathcal S^q)$\\
    $(v) \,\,g\in S^q$
\end{corollary}
\par Throughout the paper
we shall be using the convention that
$C=C(p, \alpha ,q,\beta , \dots )$ will denote a positive constant
which depends only upon the displayed parameters $p, \alpha , q,
\beta \dots $ (which sometimes will be omitted) but not  necessarily
the same at different occurrences. Furthermore, for two real-valued
functions $K_1, K_2$ we write $K_1\lesssim K_2$, or $K_1\gtrsim
K_2$, if there exists a positive constant $C$ independent of the
arguments such that $K_1\leq C K_2$, respectively $K_1\ge C K_2$. If
we have $K_1\lesssim K_2$ and $K_1\gtrsim K_2$ simultaneously, then
we say that $K_1$ and $K_2$ are equivalent and we write $K_1\asymp
K_2$.
\section{Notation and preliminaries}\label{prelim} \par
Let us start fixing some notation. If $f\in \hol (\mathbb
D)$, $f(z)=\sum_{k=0}^\infty a_kz^k$ ($z\in \mathbb D)$, we set
\begin{equation}\label{deltaN}\Delta _N(f)(z)\,=\,\sum_{k=N}^{2N-1}a_kz^k, \quad 1\le N<\infty.
\end{equation}
and
\begin{equation}\label{Df}D(f)(z)\,=\,zf^\prime (z)+ f(z)=\sum_{n=0}^\infty (n+1)a_nz^n,\quad z\in \mathbb
    D.\end{equation} In other words, $D(f)(z)=(Sf)^\prime (z)$, where $S$ is
the shift operator.
\subsection{Spaces of analytic functions in the unit disc}
\par\medskip Let
$\,0<r<1\,$ and $\,f\in \hol (\D)$. We set
$$
M_p(r,f)=\left(\frac{1}{2\pi }\int_0^{2\pi }
|f(re^{it})|^p\,dt\right)^{1/p}, \,\,\, 0<p<\infty ,
$$

$$
M_\infty(r,f)=\sup_{\vert z\vert =r}|f(z)|.
$$
For $\,0<p\le \infty $,\, the Hardy space $H^p$ consists of those
$f\in \hol(\mathbb D)$ such that $$\Vert f\Vert _{H^p}\defeq
\sup_{0<r<1}M_p(r,f)<\infty .$$
It holds that
$$
H^q \subset  H^p\, \quad 0<p <q \leq \infty\,.
$$
In particular, if $1\leq p \leq 2$ and  $f(z)=\sum_{n=0}^{\infty} a_n z^n \in \hol(\mathbb D)$ then we recall from Littlewood-Paley theory that
\begin{equation}\label{Dpp-1}
\sum_{n=0}^{\infty} (n+1)^{p-2} |a_n|^p \lesssim \|f\|_{H^p}^p \lesssim |f(0)|^p + \int_{\mathbb D} |f'(z)|^p (1-|z|^2)^{p-1} \,dA(z)
\end{equation}
where the $dA(z)$ stands for the area Lebesgue measure on $\mathbb D$\, \cite{Zy}. Notice that the inequality on the left is true even when $p\in(0,1)$. Recalling Theorem D from \cite{GGPSi},
\begin{equation}\label{Dpp-1 dyadic}
    \int_{\mathbb D} |f'(z)|^p (1-|z|^2)^{p-1} \,dA(z) \asymp |f'(0)|^p + \sum_{n=0}^{\infty} 2^{-np} \|\Delta_{2^n}(f')\|_{H^p}^p\,.
\end{equation}
Additionally, if $2<p<\infty$ then
\begin{equation}\label{Pv1}
    \|f\|_{H^p}^2 \lesssim \int_0^1  (1-r) M_p^2(f',r) \,dr \lesssim \sum_{n=0}^\infty \|\Delta_{2^n}(f)\|_{H^p}^2\,.
\end{equation}
In general, if $1<p<q<\infty$ then
\begin{equation}\label{Pv2}
    \|f\|_{H^q}^q \lesssim \int_0^1  (1-r)^{q(1-\frac 1p)} M_p^q(f',r) \,dr \,.
\end{equation}
For more information about (\ref{Pv1}), (\ref{Pv2})  see \cite{MP}. We will also  make use of a result of Hardy and Littlewood  stated as Theorem $5.11$ in \cite{Du:Hp}.
\begin{other}\label{HardyLittle}
    Let $0<p<q\leq \infty,\,\, f\in H^p, \,\,\lambda\geq p $ then
    \begin{equation}\label{HL}
        \int_0^1 (1-r)^{\lambda(\frac 1p-\frac 1q)-1} M_q^{\lambda}(f,r)\, dr < \infty\,.
    \end{equation}
\end{other}
In general, we refer to \cite{Du:Hp} for the
notation and results regarding Hardy spaces.
\par\medskip We say that an $f\in \hol (\mathbb D)$ belongs to $A^p\, (0<p<\infty)$ if
$$
\|f\|^p_{A^p}=\int_{\mathbb D} |f(z)|^p\, dA(z) <\infty \,.
$$
Moreover, due to Lemma $4.26$ in \cite{Zhu}
\begin{equation}\label{ZhuBergman1}
    \|f\|_{A^p} \asymp \|S(f)\|_{A^p} \,.
\end{equation}
According to \cite{Zhu}(p.85), the Bergman norm can be equivalently expressed in terms of the derivative as
\begin{equation}\label{ZhuBergman2}
    \|f\|_{A^p}^p \asymp |f(0)|^p + \int_{\mathbb D}  |f'(z)|^p (1-|z|^2)^{p}\,dA(z)\,.
\end{equation}
\par We say that an $f \in \hol(\mathbb D)$ belongs to weighted Dirichlet space $D^p_{\alpha},\, \alpha \in (-1,\infty),$ if $f'\in A^p_{\alpha}$\,.
For more information on Bergman spaces we suggest \cite{DS}, \cite{HKZ}, \cite{Zhu}.
\subsection{Mean Lipschitz spaces}\label{Mean Lip}
\par\medskip
If $f$ is a function which is analytic in $\mathbb D $ and has a
non-tangential limit $f(e\sp {i\theta })$ at almost every $\xi \in
\partial \mathbb D$, we define
$$
\aligned \omega _ p(\delta , f)= \sup_ {0<\vert t\vert \leq \delta
}\left (\frac{1}{2\pi }\int _{-\pi }\sp \pi \left \vert f(e\sp
{i(\theta +t)})-f(e\sp {i\theta })\right \vert \sp p\,
d\theta \right )\sp {1/p}, &\quad \delta >0,\quad\hbox{if $1\leq p<\infty $},\\
\omega _ \infty (\delta , f)= \sup_ {0<\vert t\vert \leq \delta
}\left (\operatornamewithlimits{ess.sup}_ {\theta\in [-\pi, \pi ]}
\vert f(e\sp {i(\theta +t)})-f(e\sp {i\theta })\vert\right ), &\quad
\delta >0.\endaligned $$ Then $\omega _ p(., f)$ is the integral
modulus of continuity of order $p$ of the boundary values $f(e\sp
{i\theta })$ of $f$.
For $1\le p\le \infty $ and $0<\alpha \le 1$, we let $\Lambda
^p_\alpha $ be the space of those $f\in \hol (\mathbb D)$ having a
non-tangential limit at almost every point of $\partial \mathbb D$
and so that $\omega _ p(\cdot , f)$, the integral modulus of
continuity of order $p$ of the boundary values $f(e^{i\theta })$ of
$f$, satisfies $\omega _ p(\delta , f)=\og(\delta ^{\alpha })$, as
$\delta \to 0$. When $p=\infty $ we write $\Lambda _\alpha $ instead
of $\Lambda ^\infty _\alpha $. This is the usual Lipschitz space of
order $\alpha $. More precisely, a function $f\in \hol (\mathbb D)$
belongs to $\Lambda _\alpha $ if and only if it has a continuous
extension to the closed unit disc and its boundary values satisfy a
Lipschitz condition of order $\alpha $. The corresponding \lq\lq little
oh\rq\rq \, spaces are denoted by $\lambda ^p_\alpha $.
\par
Classical results of Hardy and Littlewood (see \cite{BSS} and
\cite[Chapter\,\@5]{Du:Hp}) show that whenever $0<\alpha \le 1$ and
$1\le p\le \infty $ we have that $\Lambda ^p_{\alpha }\subset H^p$
and that
\begin{equation}\Lambda ^p_{\alpha }\,=\,\left \{ \hbox{$f$ analytic
        in $\D $: $M_ p(r, f\sp\prime )= \og \left ((1-r)\sp {\alpha
            -1}\right )$, \quad as $r\to 1$} \right\}.\end{equation} The
corresponding \lq\lq little oh\rq\rq \, spaces are denoted by
$\lambda ^p_\alpha $. It is not difficult to see that we can replace
the condition $ M_ p(r, f\sp\prime )= \og \left ((1-r)\sp {\alpha
    -1}\right )$ by  $M_ p(r, Df )= \og \left ((1-r)\sp {\alpha
    -1}\right ).$
\par If $1<p<\infty $ and $1/p<\alpha \le 1$, then
$\Lambda^p_\alpha \subset \Lambda_{\alpha -{(1/p)}}$ and, hence,
each function in $\Lambda ^p_\alpha $ has a continuous extension to
the closed unit disc \cite[p.\,\@88]{BSS}. This is not true for the
space $\Lambda ^p_{1/p}$. This follows easily noticing that the
function $f(z)=\log (1-z)$ is an unbounded function which belongs to
$\Lambda_ {1/p}\sp p$ for all $p\in (1, \infty )$. Bourdon, Shapiro,
and Sledd \cite{BSS} proved that $\Lambda^p_{1/p}\subset BMOA$ for
all $p\in (1, \infty )$. This inclusion was shown to be sharp in a
very strong sense in \cite{BGM}. Let us remark also that the
$\Lambda^p_{1/p}$-spaces form a nested scale of spaces
\begin{equation}\label{nested}
    \Lambda ^q_{1/q}\,\subset \,\Lambda ^p_{1/p}\quad 1\le q<p<\infty .
\end{equation}
\par\medskip
For $1\le p\le \infty $ and $0<\alpha \le 1$, the space $\Lambda
^p_\alpha $ is a Banach space with the norm $\Vert \cdot \Vert _{p,
    \alpha }$ defined by
$$\Vert f\Vert _{p, \alpha }\,=\,\vert f(0)\vert
\,+\,\sup_{0<r<1}(1-r)^{1-\alpha }M_p(r,f^\prime) \approx
|f(0)| + \sup_{0<r<1} (1-r)^{1-\alpha }M_p(r, Df).$$
Let us recall several
distinct characterizations of $\Lambda^p_\alpha$ spaces, (see
\cite{BSS}, \cite{Du:Hp}, and \cite{MP}).

\begin{other}\label{th:mlip}
    Suppose that $1<p<\infty$, $0<\alpha<1$ and $g\in \hol(\D)$. The
    following conditions are equivalent:
    \begin{enumerate}
        \item[(i)] $g\in \Lambda^p_\alpha .$
        \item[(ii)] $M_p(r,g')=\ogr\left(\frac{1}{(1-r)^{1-\alpha}}\right)$, as $r\to 1^-$.
        \item[(iii)] $\|\Delta_{N}(g)\|_{H^p}=
        \ogr\left(N^{-\alpha}\right)$, as $N\to\infty$.
        \item[(iv)] $\Vert \Delta_{2^n}(g
        )\Vert_{H^p}=\ogr\left(2^{-\alpha n}\right )$, as $n\to\infty $.
        \item[(v)] $\|\Delta_{N}(g')\|_{H^p}=
        \ogr\left(N^{(1-\alpha)}\right)$, as $N\to\infty$.
        \item[(vi)] $\Vert \Delta_{2^n}(g^\prime
        )\Vert_{H^p}=\ogr\left(2^{n(1-\alpha )}\right )$, as $n\to\infty $.
    \end{enumerate}
    \par By substituting \lq\lq big oh\rq\rq \, for \lq\lq little oh\rq\rq \, in (ii), (iii), (iv), (v), and (vi),
    we obtain conditions that are equivalent to the condition $g\in
    \lambda^p_\alpha $.
\end{other}
The following result, proved in \cite{GG-2025}, will be needed in our work.
\begin{other}\label{partialsumlamba} Suppose that $1<p<\infty
    $, $0<\alpha \le 1$, and $f\in \lambda ^p_\alpha $. Then:
    \par (i) Defining, for $0<r<1$, $f_r$ by
    $f_r(z)=f(rz)$\, ($\vert z\vert \le 1$), we have that
    $$\beta_{p, \alpha }(f-f_r)\,\rightarrow 0,\quad \text{as $r\to
        1$}.$$
    \par (ii)
    $S_Nf\,\rightarrow f,\,\,\text{as $N\to \infty $, in the norm of
        $\Lambda^p_\alpha $}.$
\end{other}
\subsection{Convolution of analytic functions}\label{multiplier}
\par If $f$
and $g$ are two analytic functions in the unit disc,
$$f(z)=\sum_{n=0}^\infty a_nz^n,\quad g(z)=\sum_{n=0}^\infty
b_nz^n,\quad z\in \mathbb D,$$ the convolution $f\star g$ of $f$ and
$g$ is defined by \begin{equation}\label{convdef}f\star
    g(z)\,=\,\sum_{n=0}^\infty a_nb_nz^n,\quad z\in \mathbb
    D.\end{equation} We have \begin{equation}\label{convdefint}f\star
    g(\rho re^{i\theta })\,=\,\frac{1}{2\pi }\int_0^{2\pi }f(\rho
    e^{it})g(re^{i(\theta -t)})\,dt,\quad 0<\rho <1,\,\,0<r<1,\,\,\theta
    \in \mathbb R.\end{equation} It is well known that the convolution
of a function in $H^1$ and another one in $H^p$ ($p\ge 1$) lies in
$H^p$ and that
\begin{equation*}\label{conv}\Vert f\star g\Vert_{H^p}\le \Vert
    f\Vert _{H^1}\Vert g\Vert _{H^p},\quad p\ge 1, f\in H^1, g\in
    H^p.\end{equation*} Actually, (\ref{convdefint})
implies that if $f\in H^1$ and $g\in \hol (\mathbb D)$, then
\begin{equation}\label{convh1hp} M_p(r, f\star g)\,\le \Vert f\Vert _{H^1}M_p(r, g), \quad 0<r<1.\end{equation}
Also, taking $\rho =r$ in (\ref{convdefint}), we obtain
\begin{equation}\label{Mprsquare}M_p(r^2, f\star g)\,\le M_1(r,
    f)M_p(r, g),\quad f, g\in \hol (\mathbb D),\,\,0<r<1.\end{equation}
\par\medskip We shall use in our work the following coefficient multiplier
theorem for the Hardy spaces which can be found in \cite[Vol. II,
Chapter\,\@XV,  Theorem\,\@4.\,\@14]{Zy}.
\begin{other}\label{LP-Zy} Assume that $1<p<\infty $. Then there exists a positive
    constant $A_p$ such that  if $\{ \lambda _j\} _{j=0}^\infty $ is a
    sequence of complex numbers with the property  that
    $$\vert \lambda_j\vert \le M,\,\,\,\,\sum_{j=2^k}^{2^{k+1}-1}\vert
    \lambda _{j+1}-\lambda _j\vert\,\le M,\quad j, k=0, 1, 2, \dots ,$$
    for a certain $M\ge 0$,  then the following holds: \par If $f\in
    H^p$, $f(z)=\sum_{n=0}^\infty a_nz^n$ ($z\in\mathbb D$), and
    $$h(z)=\sum_{n=0}^\infty \lambda_na_nz^n,\quad z\in \mathbb D,$$
    then $h\in H^p$ and
    $$\Vert h\Vert _{H^p}\,\le \,MA_p\Vert f\Vert _{H^p}.$$
\end{other}\par\smallskip In short, the conclusion of
Theorem\,\@\ref{LP-Zy} says that the sequence $\{ \lambda _j\}
_{j=0}^\infty $ is a coefficient multiplier from $H^p$ into itself
with multiplier-norm controlled by $M$.
\section{Proofs of the main results}\label{proofs}
\medskip
\medskip
\begin{Pf}{\it Theorem\,\@\ref{Rhaly Hardy} :}
\medskip
$(i)$ \,
First we deal with the case $1 <p \leq 2 <q $. We assume that $g\in \Lambda^q_{1/p}$. Thus,
 for any $f\in H^p$,
 \\
\begin{align*}
\|\mathcal R_{\eta}(f)\|^q_{H^q}&\lesssim \|S \circ \mathcal R_{\eta}(f)\|^q_{H^q}\\
&\lesssim \int_0^1 (1-r) M^2_q((S \circ \mathcal R_{\eta}(f))',r) dr \\
\end{align*}
due to (\ref{Pv1}). Therefore
\begin{align*}
\|\mathcal R_{\eta}(f)\|^q_{H^q} & \lesssim \int_0^1 (1-r) M^2_q((S \circ \mathcal R_{\eta}(f))',r) dr \\
& = \int_0^1 (1-r^2) M^2_q((S \circ \mathcal R_{\eta}(f))',r^2)\, 2r \, dr \\
& \lesssim \int_0^1 (1-r) M^2_q((S \circ \mathcal R_{\eta}(f))',r^2)\, dr \,.
\end{align*}
It is true that
\begin{equation}\label{Rhaly conv}
    (S \circ \mathcal R_{\eta}(f))'(z)= D(g) \star F (z)\,,\quad z \in \mathbb D
\end{equation}
where
\begin{equation}\label{F}
F(z)=f(z) \cdot \frac {1}{1-z}\,, z\in \mathbb D\,.
\end{equation}
As a consequence,
\begin{align*} \|\mathcal R_{\eta}(f)\|^q_{H^q} & \lesssim  \int_0^1 (1-r^2) M^2_q(D(g) \star F,r^2) dr \\
    &\lesssim \int_0^1 (1-r) M^2_q(D(g),r) M^2_1(F,r) dr\\
 & \asymp \int_0^1 (1-r) M^2_q(g',r) M^2_1(F,r) dr\\,.
\end{align*}
Applying the assumption $g\in \Lambda^q_{\frac{1}{p}}$
\begin{align*}
\|\mathcal R_{\eta}(f)\|^q_{H^q}
& \lesssim \int_0^1 (1-r)^{\frac{2}{p}-1} M^2_1(F,r) dr\,.
\end{align*}
Since
\begin{align*}
    M^2_1(F,r)\lesssim \frac{1}{(1-r)^{\frac{2}{q}}} M^2_q(f,r)
\end{align*}
we finally get that
\begin{align*}
    \|\mathcal R_{\eta}(f)\|^q_{H^q}
    & \lesssim \int_0^1 (1-r)^{\frac{2}{p}-1} \frac{1}{(1-r)^{\frac{2}{q}}} M^2_q(f,r) dr\\
    & = \int_0^1 (1-r)^{2(\frac{1}{p}-\frac{1}{q})-1}  M^2_q(f,r) dr\lesssim \|f\|^q_{H^p}
\end{align*}
where in the last line we have used (\ref{HL})\,.
\par\medskip Let now $1<p<q\leq 2$. Then, on account of (\ref{Dpp-1}),
\begin{align*}
    \|\mathcal R_{\eta}(f)\|^q_{H^q}&\lesssim \|S \circ \mathcal R_{\eta}(f)\|^q_{H^q}
    \lesssim \int_0^1 (1-r)^{q-1} M^q_q((S \circ \mathcal R_{\eta}(f))',r) dr \\
    &\lesssim \int_0^1 (1-r)^{q-1} M^q_q((S \circ \mathcal R_{\eta}(f))',r^2) dr\\
    & =\int_0^1 (1-r^2)^{q-1} M^q_q(D(g)\star F,r^2) dr\\
    & \leq \int_0^1 (1-r)^{q-1} M^q_q(D(g),r) M^q_1(F,r) dr\\
    &  \asymp \int_0^1 (1-r)^{q-1} M^q_q(g',r) M^q_1(F,r) dr\,.
\end{align*}
Due to the assumption $g\in \Lambda^q_{\frac{1}{p}}$
\begin{align*}
& \|\mathcal R_{\eta}(f)\|^q_{H^q}\lesssim \int_0^1 (1-r)^{\frac{q}{p}-1} M^q_1(F,r)\, dr\,.
\end{align*}
Taking a $\gamma >1$ and since $F$ is as in (\ref{F})
\begin{align*}
    M^q_1(F,r)\lesssim \frac{1}{(1-r)^{\frac{q}{\gamma p}}}\, M^q_{\gamma p}(f,r)\,.
\end{align*}
Consequently,
\begin{align*}
    \|\mathcal R_{\eta}(f)\|^q_{H^q}&\lesssim \int_0^1 (1-r)^{\frac{q}{p}-1} \frac{1}{(1-r)^{\frac{q}{\gamma p}}} M^q_{\gamma p}(f,r) dr\\
    & = \int_0^1 (1-r)^{q(\frac{1}{p}-\frac{1}{\gamma p})-1}  M^q_{\gamma p}(f,r) dr\lesssim \|f\|^q_{H^p}
\end{align*}
applying (\ref{HL}) once more.\\
\medskip
 $(ii)$: Assume that $1\leq p< \infty,\, 1<q<\infty$ and that
$$
\|\mathcal R_{\eta}(f)\|_{H^q} \lesssim \|f\|_{H^p} \,, \quad f\in H^p\,.
$$
Consider the family of test functions
 \begin{equation}\label{testf}
f_N(z)=N^{-\frac {1}{p^{'}}} \sum_{n=0}^\infty a^n_N z^n =\frac{N^{-\frac {1}{p^{'}}}}{1-a_N z}\,, \quad z\in \mathbb D
\end{equation}
for
$$a_N=1-\frac 1N\,,\,\,N=2,3...$$
Since $\sup_{N}\|f_N\|_{H^p} \lesssim 1$ we get that
$\sup_{N}\|\mathcal R_{\eta}(f_N)\|_{H^q} \lesssim 1 $ .  Furthermore, for $N=2,3...$,
$$
\mathcal R_{\eta}(f_N)(z)= \sum_{n=0}^\infty  \eta_n\, a_{n,N}\, z^n\,, \quad z\in \mathbb D
$$
where
$$a_{n,N}= N^{-\frac {1}{p^{'}}} \sum_{k=0}^n a^k_N\,, \quad  n=0,1,2,...$$
Set
 $$\beta_{n,N} = \frac{1}{a_{n,N}}  \,.$$
For each $N=2,3,..$,
\begin{equation*}\label{aN1}
 N^{-\frac {1}{p^{'}}} n  a_N^n \leq    a_{n,N} \leq N^{-\frac {1}{p^{'}}} n
\end{equation*}
so
\begin{align}\label{aN2}
    a_{n,N} \asymp N^{\frac {1}{p}} \,,\quad N\leq n \leq 2N
\end{align}
that is
\begin{align}\label{bN1}
    b_{n,N} \asymp N^{-\frac {1}{p}} \,,\quad N\leq n \leq 2N
\end{align}
Additionally,
\begin{align}\label{aN3}
|a_{n+1,N} - a_{n,N}| \asymp  N^{-\frac {1}{p^{'}}}  \,,\quad N\leq n \leq 2N
\end{align}
thus, using (\ref{aN2}) and (\ref{aN3}), we get
\begin{align}\label{bN2}
    |b_{n+1,N} - b_{n,N}|\leq  N^{-\frac {1}{p}} \,,\quad N\leq n \leq 2N
\end{align}
Set
\begin{equation*}
    \lambda_{j,\nu} =
    \begin{cases}
        b_{j,2^{\nu}}, \quad \text{if}\quad  2^\nu \leq j \leq 2^{\nu +1}-1\\
        \\
        0 ,\quad \text{otherwise}
    \end{cases}
\nu=1,2,3,...
\end{equation*}
Due to (\ref{bN1}), (\ref{bN2})
$$
|\lambda_{j,\nu}|\leq b_{2^\nu, 2^\nu}\lesssim  2^{-\frac {\nu}{p}}
$$
and
$$
\sum_{j=1}^\infty |\lambda_{j+1,\nu}- \lambda_{j,\nu}|= b_{2^\nu,2^\nu}-b_{2^{\nu+1},2^\nu}\lesssim 2^{-\frac {\nu}{p}}\,.
$$
Theorem \ref{LP-Zy} implies that
$$
\|\Delta_{2^\nu}(g)\|_{H^q}\lesssim  2^{-\frac {\nu}{p}} \|\mathcal R_{\eta}(f_{2^{\nu}})\|_{H^q} \lesssim 2^{-\frac {\nu}{p}}
$$
that is $g \in \Lambda^q_{\frac 1q}$.
\par \medskip $(iii)$ : Suppose that $1<p\leq 2,\, p<q$ and consider $g\in \lambda^q_{\frac 1p}$. Since $g\in \Lambda^q_{1/p}$ the $\mathcal R_{\eta}\in \mathcal B(H^p,H^q)$. Moreover, from Theorem \ref{partialsumlamba} (ii),
$$
\|S_{N}(g)-g\|_{\Lambda^q_{\frac 1p}}\longrightarrow 0\,, \quad N \to \infty\,.
    $$
If $N=2,3,4,...$ then we consider the finite rank operators
\begin{align*}
\mathcal R_N(f)(z)=\sum_{n=0}^N \eta_n \sum_{k=0}^n a_k z^n \,, \quad z \in \mathbb D
\end{align*}
that map $H^p$ to $H^q$. Working as in case $(i)$ we end up that
$$
\|\mathcal R_{\eta}(f)- \mathcal R_N (f)\|_{H^q}\,\, \lesssim \,\,\|S_N(g) - g\|_{\Lambda^q_p}\,\, \|f\|_{H^p} \longrightarrow 0\,,\quad N\to \infty
$$
which implies the wanted.
\par \medskip
$(iv)$ :  Assume that $1\leq p <\infty,\, 1<q<\infty$ and  $\mathcal R_{\eta}\in \mathcal K(H^p,H^q)$. Considering the test functions (\ref{testf}) we realize that $f_N \to 0, n\to \infty,$ uniformly on the compact subsets of $\mathbb D$.  The compactness of the operator implies that $\|\mathcal R_{\eta}(f_N)\|_{H^q} \to 0,\, N\to \infty$. The same steps as in case $(ii)$ lead to $\|\Delta_{2^\nu}(g)\|_{H^q}=  o (2^{-\frac {\nu}{p}}), \nu \to \infty\,,$ that is to $g\in \lambda^q_{\frac 1p}\,.$
\end{Pf}
\par\medskip
\par\medskip
\par\medskip
\begin{Pf}{\it Theorem\,\@\ref{Hg Hardy}:} (i)  First, we assume that $1<p<q \leq 2 $. If $g\in \Lambda^q_{\frac{1}{p}},\,1<p<q \leq 2, $ then for any $f\in H^p$ and due to (\ref{Dpp-1}) and (\ref{Dpp-1 dyadic})
\begin{align*}
    \|\mathcal H_g(f)\|_{H^q}& \lesssim \|\mathcal H_g(f)\|_{D^q_{q-1}} =|\mathcal H_g(f)(0)|^q + \int_{\mathbb D} |\mathcal H_{g}(f)'(z)|^q \,(1-|z|^2)^{q-1}\,dA(z)\\
    & =|\mathcal H_g(f)(0)|^q + \int_{\mathbb D} |\mathcal H_{g'}(S(f)(z)|^q \,(1-|z|^2)^{q-1}\,dA(z)\\
    & \simeq
    |\mathcal H_{g}(f)(0)|^q +|\mathcal H_{g'}(S(f))(0)|^q + \sum_{n=0}^\infty 2^{-nq} \|\Delta_{2^n}(\mathcal H_{g'}(S(f)))\|^q_{H^q}=I_1+I_2\,. \\
    \end{align*}
\par On one hand,
\begin{align*}
I_1 &=|\mathcal H_{g}(f)(0)|^q +|\mathcal H_{g'}(S(f))(0)|^q \\
&\quad\quad \lesssim    \left( |g'(0)|^q + |g''(0)|^q \right) \left( \int_0^1 |f(t)| dt \right)^q\\
    &\quad \quad \quad \lesssim
    \left( |g'(0)|^q + |g''(0)|^q \right) \|f\|^q_{H^p}
\end{align*}
where in the last line we have used Fejer-Riesz inequality \cite{Du:Hp}\,.\\
\par On the other hand,
Lemma 7 \cite{GGPSi}  implies that
\begin{align*}
    \|\Delta_{2^n}(\mathcal H_{g'}(S(f)))\|^q_{H^q} &\lesssim \left( \int_0^1 t^{2^{n-1}+1}|f(t)| dt \right)^q\, \|\Delta_{2^n}(g'')\|^q_{H^q},,\quad n\geq 3\,.
\end{align*}
Therefore,
\begin{align*}
    \sum_{n=3}^\infty 2^{-nq} & \|\Delta_{2^n}(\mathcal H_{g'}(S(f)))\|_{H^q} \lesssim
    \sum_{n=3}^\infty 2^{-nq}\, \left( \int_0^1 t^{2^{n-1}+1}|f(t)| dt \right)^q\, \|\Delta_{2^n}(g'')\|^q_{H^q}\,.
\end{align*}
Employing the assumption $g\in \Lambda^q_{\frac 1p},$ as stated equivalently in part (vi) of Theorem \ref{th:mlip},
\begin{align*}
\sum_{n=3}^\infty 2^{-nq} & \|\Delta_{2^n}(\mathcal H_{g'}(f))\|_{H^q} \lesssim \sum_{n=3}^\infty 2^{-nq}\, \left( \int_0^1 t^{2^{n-1}+1}|f(t)| dt \right)^q\, 2^{nq(2-\frac 1p)}\\
&   =\sum_{n=3}^\infty  2^{nq(1-\frac 1p)}\, \left( \int_0^1 t^{2^{n-1}+1}|f(t)| dt \right)^q
 = \sum_{n=0}^\infty  2^{(n+3)q(1-\frac 1p)}\, \left( \int_0^1 t^{2^{n+1}+1}|f(t)| dt \right)^q\\
    & \asymp \sum_{n=0}^\infty  2^{(n+1)q(1-\frac 1p)}\, \left( \int_0^1 t^{2^{n+1}+1}|f(t)| dt \right)^{q-p} \left( \int_0^1 t^{2^{n+1}+1}|f(t)| dt \right)^{p}\,.
\end{align*}
An application of Holder and Fejer-Riesz inequalities implies that
\begin{align*}
\sum_{n=3}^\infty 2^{-nq} & \|\Delta_{2^n}(\mathcal H_{g'}(f))\|_{H^q} \lesssim \sum_{n=0}^\infty  2^{(n+1)q(1-\frac 1p)}\, \left( \int_0^1 t^{2^{n+1}+1}|f(t)| dt \right)^{q-p} \left( \int_0^1 t^{2^{n+1}+1}|f(t)| dt \right)^{p}\\
    &\lesssim \sum_{n=0}^\infty   2^{(n+1)q(1-\frac 1p)}\,
    \left( \int_0^1 t^{p'2^n} dt \right)^{\frac {q-p}{p'}} \left( \int_0^1 t^{2^{n+1}+1}|f(t)| dt \right)^{p} \|f\|^{q-p}_{H^p}\\
    & \lesssim \sum_{n=0}^\infty   2^{(n+1)q(1-\frac 1p)} \frac{1}{2^{(n+1){\frac{q-p}{p'}}}}\,\left( \int_0^1 t^{2^{n+1}+1}|f(t)| dt \right)^{p} \|f\|^{q-p}_{H^p}\\
    & = \sum_{n=0}^\infty  2^{(n+1){\frac{p}{p'}}} \,\left( \int_0^1 t^{2^{n+1}+1}|f(t)| dt \right)^{p} \|f\|^{q-p}_{H^p}\\
    & = \sum_{n=0}^\infty  2^{(n+1)(p-1)} \,\left( \int_0^1 t^{2^{n+1}+1}|f(t)| dt \right)^{p} \|f\|^{q-p}_{H^p}\\
    & = \sum_{n=0}^\infty  2^{n(p-2)} 2^n \,\left( \int_0^1 t^{2^{n+1}+1}|f(t)| dt \right)^{p} \|f\|^{q-p}_{H^p}\,.
\end{align*}
It is true that
\begin{align*}
    2^n \,\left( \int_0^1 t^{2^{n+1}+1}|f(t)| dt \right)^{p} \leq \sum_{k=2^n}^{2^{n+1}-1} \,\left( \int_0^1 t^{k+1}|f(t)| dt \right)^{p}\,.
\end{align*}
Thus
\begin{align*}
\sum_{n=3}^\infty 2^{-nq} & \|\Delta_{2^n}(\mathcal H_{g'}(f))\|_{H^q}  \lesssim \sum_{n=0}^\infty  2^{n(p-2)} \sum_{k=2^n}^{2^{n+1}-1} \,\left( \int_0^1 t^{k+1}|f(t)| dt \right)^{p} \|f\|^{q-p}_{H^p}\\
    & \asymp \sum_{n=1}^\infty (n+1)^{p-2} \left( \int_0^1 t^{n+1}|f(t)| dt \right)^{p} \|f\|^{q-p}_{H^p} 
\end{align*}
Consider the sublinear operator $\tilde{H}(f)(z)= \int_0^1 \frac{|f(t)|}{1-tz}\,dt,\, z\in \mathbb D$. Recalling  Theorem 5 and Theorem C from \cite{GGPSi} we are allowed to say that
\begin{align*}
    \sum_{n=3}^\infty 2^{-nq} & \|\Delta_{2^n}(\mathcal H_{g'}(f))\|_{H^q}  \lesssim \|\tilde{H}(f)\|^{p}_{H^p} \|f\|^{q-p}_{H^p} \lesssim \|f\|^p_{H^p} \|f\|^{q-p}_{H^p} = \|f\|^q_{H^p}\,.
\end{align*}
Combining  all the above we get that
$$
I_2 \lesssim \|f\|^q_{H^p}\,.
$$
Joining the latter with the estimation of $I_1$ we get the desired.
\par Now, assume that $1<p\leq 2 <q $. Relation (\ref{Pv1}) implies that
\begin{align*}
\|\mathcal H_g (f) \|^2_{H^q} &\lesssim \sum_{n=0}^\infty \|\Delta_{2^n} (\mathcal H_g (f))\|_{H^q}^2\,.
\end{align*}
Employing Lemma 7 \cite{GGPSi}, the assumption $g\in \Lambda^q_{\frac 1p},$ as stated equivalently in part (vi) of Theorem \ref{th:mlip}, Holder and Fejer-Riesz inequalities
\begin{align*}
    \sum_{n=3}^\infty & \|\Delta_{2^n} (\mathcal H_g)\|_{H^q}^2 \lesssim
\sum_{n=3}^\infty  \left( \int_0^1 t^{2^{n-1}+1}|f(t)| dt \right)^2\, \|\Delta_{2^n}(g')\|^2_{H^q}\\
& \lesssim \sum_{n=0}^\infty 2^{2n(1-\frac 1p)} \left( \int_0^1 t^{2^{n+1}+1}|f(t)| dt \right)^{2-p}
 \left( \int_0^1 t^{2^{n+1}+1}|f(t)| dt \right)^p\\
 & \lesssim \sum_{n=0}^\infty 2^{2n(1-\frac 1p)} 2^{-n(\frac {2-p}{p'})}\left( \int_0^1 t^{2^{n+1}+1}|f(t)| dt \right)^p\,\,\|f\|_{H^p}^{q-p}\\
 & = \sum_{n=0}^\infty 2^{n(p-1)} \left( \int_0^1 t^{2^{n+1}+1}|f(t)| dt \right)^p \,\,\|f\|_{H^p}^{q-p}\\
 & \lesssim \sum_{n=0}^\infty  2^{n(p-2)} \sum_{k=2^n}^{2^{n+1}-1} \,\left( \int_0^1 t^{k+1}|f(t)| dt \right)^{p} \|f\|^{q-p}_{H^p}\\
 & \lesssim \sum_{n=1}^\infty (n+1)^{p-2} \left( \int_0^1 t^{n}|f(t)| dt \right)^{p} \|f\|^{q-p}_{H^p} \lesssim \|f\|^p_{H^p} \|f\|^{q-p}_{H^p} =\|f\|^q_{H^p}\,.
\end{align*}
\par (ii) Assume that $1\leq p < \infty,\, 1<q<\infty$ and  that $\mathcal H_g\in \mathcal B(H^p, H^q)$ that is
$$
\|\mathcal H_g (f)\|_{H^q} \lesssim \|f\|_{H^p} \,, \quad f\in H^p\,.
$$
For $N=2,3...$ let $a_N=1-\frac 1N$ and
$$
f_N(z)=N^{-\frac {1}{p^{'}}} \sum_{n=0}^\infty a^n_N z^n =\frac{N^{-\frac {1}{p^{'}}}}{1-a_N z}\,, \quad z\in \mathbb D.
$$
Since $
\|f_N\|_{H^p} \lesssim 1$ for all $N$ it holds that $\|\mathcal H_g (f_N)\|_{H^q} \lesssim 1 $ for all $N$.
Additionally
$$
\mathcal H_g (f_N)(z)= \sum_{n=0}^\infty  (n+1)\eta_{n+1} a_{n,N} z^n
$$
where $$a_{n,N}= N^{-\frac {1}{p^{'}}} \sum_{k=0}^\infty \frac{a^k_N}{n+k+1} \,, \quad  n=0,1,2,...$$
Moreover, if $ N \leq n \leq 2N$ then
\begin{align*}
    a_{n,N}\leq N^{-\frac {1}{p^{'}}} \frac {1}{n+1} \frac{1}{1-a_N} \asymp N^{-\frac {1}{p^{'}}}
\end{align*}
and
\begin{align*}
    a_{n,N}\geq N^{-\frac {1}{p^{'}}} \sum_{k=N}^{2N}  \frac{a^k_N}{n+k+1} 
    \gtrsim N^{-\frac {1}{p^{'}}}
\end{align*}
As a consequence,
\begin{align*}
    |a_{n+1,N}-a_{n,N}|&\leq N^{-\frac {1}{p^{'}}} \left|  \sum_{k=0}^\infty  a^k_N \left(\frac{1}{n+k+2} - \frac{1}{n+k+1}\right) \right|\\
    & = N^{-\frac {1}{p^{'}}}  \sum_{k=0}^\infty  a^k_N \frac{1}{(n+k+1)(n+k+2)}\\
    & \lesssim  N^{-\frac {1}{p^{'}}} \frac{1}{N} \sum_{k=0}^\infty  a^k_N = N^{-\frac {1}{p^{'}}}
\end{align*}

If we set $$\beta_{n,N} = \frac{1}{a_{n,N}} $$ then, for $ N \leq n \leq 2N$,
$$
\beta_{n,N} \lesssim N^{\frac {1}{p^{'}}}
$$
and
$$
|\beta_{n+1,N}-\beta_{n,N}| \lesssim N^{\frac {1}{p^{'}}}\,.
$$

Let an $\nu \in \mathbb N $ and set
\begin{equation*}
    \lambda_{j,\nu} =
    \begin{cases}
        \beta_{j,2^{\nu}}, \quad \text{if}\quad  2^\nu \leq j \leq 2^{\nu +1}-1\\
        \\
        0 ,\quad \text{otherwise}
    \end{cases}
\end{equation*}
Then,
$$
|\lambda_{j,\nu}|\leq \beta_{2^\nu, 2^\nu}\lesssim 2^{\frac {\nu}{p'}}
$$
and
$$
\sum_{j=1}^\infty |\lambda_{j+1,\nu}- \lambda_{j,\nu}|= \beta_{2^\mu,2^\nu}-\beta_{2^{\nu+1},2^\nu}\lesssim  2^{\frac {\nu}{p'}}
$$
Due to Theorem \ref{LP-Zy}
$$
\|\Delta_{2^\nu}(g')\|_{H^q}\lesssim A_p\,  2^{\frac {\nu}{p'}} \|\mathcal H_g (f_{2^{\nu}})\|_{H^q} \lesssim A_p\,  2^{\nu (1-\frac {1}{p})}
$$
which implies that $g \in \Lambda^q_{\frac 1p} $ according to part (vi) of Theorem \ref{th:mlip}.
\par $(iii)$  Assume that $ 1<p<q \leq 2,$, $g\in\lambda^q_{\frac 1p}$ and consider a sequence $\{f_N\}_{N=1}^\infty$ in $H^p$ such that $\sup_{N}\|f_N\|_{H^p}<\infty$ and that $f_N \to 0, N\to \infty,$ uniformly on compact subsets of $\mathbb D$.
As stated by Lemma 10 \cite{GGPSi} we want to prove that
$$\|\mathcal H_{g}(f_N)\|_{H^q} \to 0,\quad  N \to \infty\,.$$
So, working as in case $(i)$
\begin{align*}
    \|\mathcal H_g(f_N)\|^q_{H^q} \lesssim & \|\mathcal H_g(f_N)\|^q_{D^q_{q-1}} \asymp
    |\mathcal H_{g}(f_N)(0)|^q +|\mathcal H_{g'}(S(f_N))(0)|^q \\ & \quad \quad\quad\quad\quad\quad \quad\quad+ \sum_{n=0}^\infty 2^{-nq} \|\Delta_{2^n}(\mathcal H_{g'}(S(f_N)))\|^q_{H^q}\,.
\end{align*}
Recall that
\begin{align*}
    |\mathcal H_{g}(f_N)(0)|^q +|\mathcal H_{g'}(S(f_N))(0)|^q
     \lesssim   \left( |g'(0)|^q + |g''(0)|^q \right) \left( \int_0^1 |f_N(t)| dt \right)^q \to 0,\quad N \to \infty
\end{align*}
as a consequence of Lemma 9 \cite{GGPSi}\,.
\par On the other hand, the assumption that $g\in \lambda^q_p$ implies that if we choose an $\epsilon >0$ then there exists $n_0\in \mathbb N, n_0 >3,$ such that
$$
\frac {\|\Delta_{2^n}(g'')\|_{H^q}}{2^{n(1-\frac 1p)}}<\epsilon, \quad n\geq n_0 \,.
$$
Thus
\begin{align*}
    \sum_{n=0}^\infty 2^{-nq} & \|\Delta_{2^n}(\mathcal H_{g'}(S(f_N)))\|_{H^q}\\
    & =  \sum_{n=0}^{n_0-1} 2^{-nq}  \|\Delta_{2^n}(\mathcal H_{g'}(S(f_N)))\|_{H^q} +  \sum_{n=n_0}^\infty 2^{-nq}  \|\Delta_{2^n}(\mathcal H_{g'}(S(f_N)))\|_{H^q} \\
    &=I_1 + I_2 \,.
\end{align*}
Using again  Lemma 9 \cite{GGPSi} we get that
\begin{align*}
I_1=\sum_{n=0}^{n_0-1} 2^{-nq}  \|\Delta_{2^n}(\mathcal H_{g'}(S(f_N)))\|_{H^q} \to 0\,, \quad N\to \infty \,.
\end{align*}
Now, employing Lemma 7 \cite{GGPSi} and the same steps as in case (i)
\begin{align*}
I_2 & = \sum_{n=n_0}^\infty 2^{-nq}  \|\Delta_{2^n}(\mathcal H_{g'}(S(f_N)))\|_{H^q}\\
& \lesssim \sum_{n=n_0}^\infty 2^{-nq}\, \left( \int_0^1 t^{2^{n-1}+1}|f_N(t)| dt \right)^q\, \|\Delta_{2^n}(g'')\|^q_{H^q}\\
    & \lesssim \epsilon  \sum_{n=n_0}^\infty  2^{nq(1-\frac 1p)}\, \left( \int_0^1 t^{2^{n-1}+1}|f_N(t)| dt \right)^q \\
    & \lesssim \epsilon \|\tilde{H}(f_N)\|^{p}_{H^p} \|f_N\|^{q-p}_{H^p}  \lesssim \epsilon \|f_N\|^p_{H^p} \|f_N\|^{q-p}_{H^p} \lesssim \epsilon\,.
\end{align*}
Since  $\epsilon>0$ is arbitrary it holds that
$\lim_{N\to \infty} \|\mathcal H_g(f_N)\|^q_{H^q} =0 $
\par The same structure as above combined with the steps used in the proof of case $(i)$ when $1<p\leq 2< q$  give the wanted. We omit the details.
\par $(iv)$ Let  $1\leq p < \infty,\, 1<q<\infty$. We suppose that the operator is compact. We consider the test functions
$$
f_N(z)=N^{-\frac {1}{p^{'}}} \sum_{n=0}^\infty a^n_N z^n =\frac{N^{-\frac {1}{p^{'}}}}{1-a_N z}\,, \quad z\in \mathbb D\,,
$$
where $a_N=1-\frac 1N, N=2,3...$. Since $f_N \to 0, N\to \infty,$ uniformly on the compact subsets of $\mathbb D$ and $\sup_N \|f_N\|_{H^p}<\infty\,,$  the assumption that the operator is compact and Lemma 10 \cite{GGPSi} implies that $\lim_{N\to \infty} \|\mathcal H_g(f_N)\|_{H^q} =0\,.$
Now, thinking the same way as in part $(ii)$ we reach the wanted.
\end{Pf}
\\

\begin{Pf} {\it  Theorem\,\@\ref{Rhaly Bergman}:} $(i)$ Suppose that $1<p<q<\infty,\, \frac{2}{p}-\frac{1}{q}<1$ and
$g\in \Lambda^q_{\frac{2}{p}-\frac{1}{q}}$.  Then
\begin{align*}
    \|\mathcal R_{\eta}(f)\|_{A^q}^q & \asymp \|S \circ \mathcal R_{\eta}(f) \|_{{A^q}} \asymp \int_{\mathbb D} |(S \circ \mathcal R_{\eta}(f))'(z)|^q (1-|z|^2)^q dA(z)\\
    &= \int_0^1 M^q_q((S \circ \mathcal R_{\eta}(f))',r) \,(1-r^2)^q \, dr
    \lesssim \int_0^1 M^q_q((S \circ \mathcal R_{\eta}(f))',r) \,(1-r)^q \, dr\\
    & = 2 \int_0^1 M^q_q((S \circ \mathcal R_{\eta}(f))',r^2) \,(1-r^2)^q \,r\, dr\lesssim \int_0^1 M^q_q((S \circ \mathcal R_{\eta}(f))',r^2) \,(1-r)^q \,r\, dr \,.
    \end{align*}

Employing (\ref{Rhaly conv}),(\ref{Mprsquare}), (\ref{F}) and Holder's inequality
\begin{align*}
\|\mathcal R_{\eta}(f)\|_{A^q}^q    &\lesssim \int_0^1 M^q_q(D(g)\star F, r^2) (1-r)^q \,r \,dr
    \lesssim \int_0^1 M^q_q(D(g),r) M^q_1(F, r) (1-r)^q \,r \,dr\\
    & \lesssim \int_0^1 M^q_q(g',r)  M^q_1(F, r) (1-r)^q \,r \,dr
     \lesssim  \int_0^1 M^q_q(g',r)  M^q_p(f, r) \frac{(1-r)^q}{(1-r)^{\frac {q}{p}}} \,dr\\
     &= \int_0^1 M^q_q(g',r) M^p_p(f, r)  M^{q-p}_p(f, r) (1-r)^{q(1-
        \frac{1}{p})} \,dr\\
    &\lesssim \int_0^1 M^q_q(g',r) M^p_p(f, r)  \frac{(1-r)^{q(1-
            \frac{1}{p})}}{(1-r)^{\frac {q-p}{p}}} \,dr\, \|f\|^{q-p}_{A^p}\\
        & \lesssim \sup_{r} M^q_q(g',r)(1-r)^{q\big(1-(\frac 2p -\frac 1q)\big)} \,\,\|f\|_{A^p}^p\,\,\|f\|_{A^p}^q \leq \|g\|^q_{q,(\frac 2p -\frac 1q)}\,\,\|f\|^q_{A^p}
\end{align*}
which leads to the boundedness of the operator.\\
\par On the opposite direction, we assume that $\mathcal R_{\eta}\in \mathcal B(A^p,A^q), 1<p<q<\infty,\, \frac{2}{p}-\frac{1}{q}<1\,.$
Let $$  f_N (z)= \frac{1}{N^{\frac {2}{p'}}} \frac{1}{(1-a_N z)^2} \,,\quad N=2,3,.... $$
for which  holds that $\|f_N\|_{A^p} \lesssim 1$ for every $N$.

\par We set
$$a_{n,N}=  \frac{1}{N^{\frac {2}{p'}}} \sum_{k=0}^n (k+1) a_N^k\,. $$ then  for $N\leq n \leq 2N,$
$$
a_{n,N} \asymp N^{\frac 2p}\quad \text{and} \quad
|a_{n+1, N}-a_{n,N}| \asymp   N^{1-\frac {2}{p'}}\,.
$$
\par If $$b_{n,N}=\frac{1}{a_{n,N}}$$ then, for $N\leq n \leq 2N$,
$$
b_{n,N} \asymp N^{-\frac 2p}\quad \text{and} \quad
|b_{n+1, N}-b_{n,N}| \asymp   N^{1-\frac {2}{p'} - \frac{4}{p}}= N^{-(1+ \frac{2}{p})} \leq  N^{-\frac{2}{p}}\,.
$$
Take an $\nu \in \mathbb N $ and set
\begin{equation*}
    \lambda_{j,\nu} =
    \begin{cases}
        \beta_{j,2^{\nu}}, \quad \text{if}\quad  2^\nu \leq j \leq 2^{\nu +1}-1\\
        \\
        0 ,\quad \text{otherwise}
    \end{cases}
\end{equation*}
Therefore
$$
|\lambda_{j,\nu}|\leq \beta_{2^\nu, 2^\nu}\lesssim  2^{-\frac {2\nu}{p}}
$$
and
$$
\sum_{j=1}^\infty |\lambda_{j+1,\nu}- \lambda_{j,\nu}|= \beta_{2^\nu,2^\nu}-\beta_{2^{\nu+1},2^\nu}\lesssim  2^{-\frac {2\nu}{p}}
$$
Therefore, using Theorem \ref{LP-Zy}
\begin{align*}
\|\Delta_{2^\nu}(g)\|_{H^q} & \lesssim  2^{-\frac {2\nu}{p}} \|\Delta_{2^\nu} \mathcal (R_{\eta}(f_{2^{\nu}}))\|_{H^q}\\
\\
& \asymp  2^{- \nu(\frac {2}{p}-\frac 1q)} \| \Delta_{2^\nu}\mathcal (R_{\eta}(f_{2^{\nu}}))\|_{A^q} \lesssim 2^{- \nu(\frac {2}{p}-\frac 1q)} \|\mathcal R_{\eta}(f_{2^{\nu}}))\|_{A^q}\\
\\
&  \lesssim 2^{- \nu(\frac {2}{p}-\frac 1q)}
\end{align*}
\\
which implies that $g \in \Lambda^q_{\frac {2}{p}-\frac 1q}$. \\
\par $(ii) :$  Assume that $1<p<q<\infty,\, \frac{2}{p}-\frac{1}{q}<1$ and $g\in\lambda^q_{\frac {2}{p}-\frac 1q}.$ Since $g\in\Lambda^q_{\frac {2}{p}-\frac 1q}$
the $\mathcal R_{\eta}\in \mathcal B(A^p,A^q)$. Moreover, from Theorem \ref{partialsumlamba} (ii), we are aware that
$$
\|S_{N}(g)-g\|_{\Lambda^q_{\frac {2}{p}-\frac 1q}}\longrightarrow 0\,, \quad N \to \infty,
$$
If $N=2,3,4,...$ then we consider the finite rank operators
\begin{align*}
    \mathcal R_N(f)(z)=\sum_{n=0}^N \eta_n \sum_{k=0}^n a_k z^n \,, \quad z \in \mathbb D
\end{align*}
that map $A^p$ to $A^q$. Working as in case $(i)$ we end up that
$$
\|\mathcal R_{\eta}(f)- \mathcal R_N (f)\|_{A^q}\,\, \lesssim \,\,\|S_N(g) - g\|_{\Lambda^q_{\frac {2}{p}-\frac 1q}}\,\, \|f\|_{A^p} \longrightarrow 0\,,\quad N\to \infty\,.
$$
which implies the wanted.
\par Thinking the other way around, we assume that $\mathcal R_{\eta}\in \mathcal K(A^p,A^q), 1<p<q<\infty,\, \frac{2}{p}-\frac{1}{q}<1$. Next, we consider the family of functions $$  f_N (z)= \frac{1}{N^{\frac {2}{p'}}} \frac{1}{(1-a_N z)^2} \,,\quad N=2,3,.... ,$$
for which is true that $\sup_N \|f_N\|_{A^p} \lesssim 1$  and $f_N \to 0, N\to \infty,$  uniformly on the compact subsets of the unit disc. Therefore
$\mathcal R_{\eta}(f_N) \to 0, N\to \infty$ in $A^q$. Working as in the part of the proof of (i)  related to the necessary conditions and employing Theorem \ref{LP-Zy} we get that $g\in\lambda^q_{\frac {2}{p}-\frac 1q}.$
\end{Pf}

\medskip
\begin{Pf}{\it Theorem\,\@\ref{Rhaly Dpa Dpb}:} $(i)$ Let $g\in \Lambda^p_s,\, s=\frac{a-b+1}{p},$ then for every $f\in D^p_a$
\begin{align*}
    \|\mathcal R_{\eta}(f) \|^p_{D^p_b} &= |\mathcal R_{\eta}(f)(0)|^p +
    \int_{\mathbb D} |\mathcal R_{\eta}(f)'(z) |^p \, (1-|z|^2)^{b}\, dA(z)\\
    & \asymp |\mathcal R_{\eta}(f)(0)|^p +  |\mathcal R_{\eta}(f)'(0)|^p +
    \int_{\mathbb D} |\mathcal R_{\eta}(f)''(z) |^p \, (1-|z|^2)^{b+p}\, dA(z)\\
    & = |\eta_0|^p |a_0|^p + |\eta_1|^p (|a_0|+|a_1|)^p + \int_0^1 (1-r^2)^{b+p} \, M^p_p(\mathcal R_{\eta}(f)'',r) \,r\,dr\,. \\
\end{align*}
We set $F(z)=\frac{f(z)}{1-z}\,,\, z\in \mathbb D\,.$ Applying that $g\in \Lambda^p_s$
\begin{align*}
    \int_0^1 & (1-r^2)^{b+p} \, M^p_p(\mathcal R_{\eta}(f)'',r) \,r\,dr \lesssim \int_0^1 (1-r^2)^{b+p} \, M^p_p(g'\star F',r) \,r\,dr\\
    &\,\,\,\, \lesssim \int_0^1 (1-r^2)^{b+p} \, M^p_p(g',r)\,M^p_1(F',r)\,dr \lesssim \int_0^1 (1-r^2)^{b+sp} \,M^p_1(F',r)\,dr\\
    & \lesssim \int_0^1 (1-r^2)^{b+sp} \,M^p_1(\phi,r)\,dr + \int_0^1 (1-r^2)^{b+sp} \,M^p_1(\psi,r)\,dr\doteqdot I_1 + I_2
    \end{align*}
where
$ \phi(z)=\frac{f'(z)}{1-z}\,,\,\, \psi(z)=\frac{f(z)}{(1-z)^2}$\,\,.

Now,
\begin{align*}
    I_1 &\doteqdot \int_0^1 (1-r)^{b+sp} \,M^p_1(\phi,r)\,dr \lesssim \int_0^1 (1-r)^{b+sp-p} \,M^p_p(f',r)\,dr \\
    & = \int_0^1 (1-r)^{a} \,M^p_p(f',r)\,dr
\end{align*}
and, for $\gamma >1$ such that $a>p-2+\frac{1}{\gamma},$
\begin{align*}
    I_2 &\doteqdot \int_0^1 (1-r)^{b+sp} \,M^p_1(\psi,r)\,dr \lesssim  \int_0^1 (1-r)^{b+sp} \frac{1}{(1-r)^{p-\frac{1}{\gamma}}} \,M^p_{\gamma p}(f,r)\,dr\\
    & = \int_0^1 (1-r)^{a+1-p-\frac{1}{\gamma}} \,M^p_{\gamma p}(f,r)\,dr
    \asymp  \int_0^1 (1-r)^{a+1-\frac{1}{\gamma}} \,M^p_{\gamma p}(f',r)\,dr\,.
\end{align*}
Since
$$
M_{\gamma p}(f',r)\lesssim \frac{1}{(1-r)^{\frac{1}{p}-\frac{1}{\gamma p}}} M_{ p}(f',r)
$$
we get that
\begin{align*}
        I_2 & \lesssim \int_0^1 (1-r)^{a} \,M^p_{ p}(f',r)\,dr\,.
\end{align*}
\\
\\
On the other hand, assume that $\mathcal R_{\eta} \in \mathcal B(D^p_a,D^p_b)$. If
$$
g_N(z)= \frac{1}{N^{2-\frac{a+2}{p}}}\,\, \frac{a_N z}{1-a_N z}\,,\, z \in \mathbb D
$$
where $a_n=1-\frac 1N, \,\,N=2,3,...$, then
\begin{align*}
    \|g_N\|^p_{D^p_a}= \int_{\mathbb D} (1-|z|^2)^a |g'_N (z)|^p \, dA(z) \lesssim 1
\end{align*}
so $\|\mathcal R(g_N)\|_{D^p_a} \lesssim 1$ for every $N=2,3,...$
\\
\par We define $a_{n,N}= \frac{1}{N^{2-\frac{a+2}{p}}} \sum_{k=0}^n a^k_N$ then, for $N\leq n \leq 2N$,
\begin{align*}
&   a_{n,N} \asymp N^{-(1-\frac{a+2}{p})}\\
& a_{n+1,N}-a_{n,N}=N^{-(2-\frac{a+2}{p})} a^{n+1}_N \asymp  N^{-(2-\frac{a+2}{p})}\,.
\end{align*}
Set $b_{n,N}=\frac{1}{a_{n,N}}$ thus
\begin{align*}
    &  b_{n,N} \asymp N^{1-\frac{a+2}{p}}
    &   b_{n+1,N} - b_{n,N} \asymp N^{-\frac{a+2}{p}}\,,\quad n\leq n \leq 2N\,.\\
\end{align*}
Take a $\nu \in \mathbb N $ and set
\begin{equation*}
    \lambda_{j,\nu} =
    \begin{cases}
        \beta_{j,2^{\nu}}, \quad \text{if}\quad  2^\nu \leq j \leq 2^{\nu +1}-1\\
        \\
        0 ,\quad \text{otherwise}
    \end{cases}
\end{equation*}
Therefore
$$
|\lambda_{j,\nu}|\leq \beta_{2^\nu, 2^\nu}\lesssim  2^{\nu (1-\frac{a+2}{p})}
$$
and
$$
\sum_{j=1}^\infty |\lambda_{j+1,\nu}- \lambda_{j,\nu}|= \beta_{2^\mu,2^\nu}-\beta_{2^{\nu+1},2^\nu}\lesssim 2^{\nu (1-\frac{a+2}{p})}
$$
Then
\begin{align*}
\|\Delta_{2^\nu}(D(g))\|_{H^p}& \lesssim 2^{\nu (1-\frac{a+2}{p})} \|\Delta_{2^\nu}(D\mathcal R_{\eta}(f_{2^{\nu}}))\|_{H^p} \\
& \\
& \asymp 2^{\nu \frac{1+b}{p}} 2^{\nu (1-\frac{a+2}{p})} \|\Delta_{2^\nu}(D\mathcal R_{\eta}(f_{2^{\nu}}))\|_{A^p_b} \\
& \\
& \lesssim 2^{\nu (1-\frac{a-b+1}{p})}
 \|\mathcal R_{\eta}(f_{2^{\nu}})\|_{D^p_b} \lesssim 1 \,, \quad \nu\in \mathbb N
\end{align*}
\\
which implies that $g \in \Lambda^p_{s}$ for $s= \frac{a-b+1}{p}$\,\,.\\
\par $(ii)$  The proof is similar as before. We omit the details.
\quad \quad \end{Pf}
\par\medskip
\begin{Pf} {\it Theorem\,\@\ref{Hg on Sp}:}
$(ii) \Rightarrow (i)$  It results  immediately.\\
\\
$(i) \Rightarrow (iii)$  Assume that $1<p<q<\infty$ and that
$$
\|\mathcal H_g (f)\|_{S^q} \lesssim \|f\|_{S^p}\,, \quad f\in S^p
$$
then the
$$
\mathcal H_g(1)=  \sum_{n=0}^\infty \eta_{n+1} z^n = \frac{g(z)-g(0)}{z} \in S^q
$$
that is $g\in S^q$.
\\
\\
 $(iii) \Rightarrow (i)$  Suppose that $1<p<q<\infty$ and that
 $$g\in S^q \Leftrightarrow g'\in H^q\,.$$
  It is enough to prove
$$
\|D(\mathcal H_g(f))\|_{H^q} \lesssim \|f\|_{S^p}\,, \quad f\in S^p\,.
$$
Since
$$
\mathcal H_g(f)(z)= \sum_{n=0}^\infty (n+1)\eta_{n+1} \sum_{k=0}^\infty \frac{a_k}{n+k+1}\, z^n\,, \quad z\in \mathbb D
$$
then
\begin{align*}
D(\mathcal H_g(f))(z)& =\sum_{n=0}^\infty (n+1)^2 \eta_{n+1} \sum_{k=0}^\infty \frac{a_k}{n+k+1}\, z^n\\
& = \big(D(\mathcal H(f)) \star g'\big) (z)\,, \quad z\in \mathbb D\,.
\end{align*}
Set $\lambda_{n}=(n+1)\sum_{k=0}^\infty \frac{a_k}{n+k+1}$ then
\begin{align*}
|\lambda_{n}|\leq (n+1)\sum_{k=0}^\infty \frac{|a_k|}{n+k+1}\leq \sum_{k=0}^\infty |a_k|\lesssim \|f\|_{S^p}
\end{align*}
where the last inequality is justified by the fact that
$$
f \in S^p \Longleftrightarrow  f' \in H^p \subseteq H^1
$$
and by the Fejer-Riesz inequality applied on $f'$ that is
$$ \sum_{k=0}^\infty |a_{k+1}| \lesssim \|f'\|_{H^1}\,.$$
Moreover, for each $n$,
\begin{align*}
    \sum_{j=2^n}^{2^{n+1}-1}&|\lambda_{j+1}-\lambda_j| =\sum_{j=2^n}^{2^{n+1}-1} \left|(j+2)\sum_{k-0}^\infty \frac{a_k}{j+k+2}-(j+1)\sum_{k=0}^\infty \frac{a_k}{j+k+1}\right|\\
    &= \sum_{j=2^n}^{2^{n+1}-1}\left|\sum_{k-0}^\infty \frac{a_k}{j+k+2} + (j+1) \left(\sum_{k-0}^\infty \frac{a_k}{j+k+2}- \sum_{k=0}^\infty \frac{a_k}{j+k+1}\right) \right|\\
    &= \sum_{j=2^n}^{2^{n+1}-1}\left|\sum_{k-0}^\infty \frac{a_k}{j+k+2} + (j+1) \sum_{k-0}^\infty a_k \left( \frac{1}{j+k+2}- \frac{1}{j+k+1}\right)  \right|\\
    & =\sum_{j=2^n}^{2^{n+1}-1}\left|\sum_{k-0}^\infty \frac{a_k}{j+k+2} + (j+1) \sum_{k-0}^\infty \frac{a_k} {(j+k+2) (j+k+1)} \right|\\
    & \leq \sum_{j=2^n}^{2^{n+1}-1} \left( \sum_{k-0}^\infty \frac{|a_k|}{j+k+2}
    + (j+1) \sum_{k-0}^\infty \frac{|a_k|} {(j+k+2) (j+k+1)} \right)\\
    &\lesssim \sum_{j=2^n}^{2^{n+1}-1} \frac 1j \,\, \sum_{k=0}^\infty |a_k| \lesssim \|f\|_{S^p}\,.
\end{align*}
Therefore, applying  Theorem \ref{LP-Zy}
\begin{align*}
\|\mathcal H_g(f)\|_{S^q} \asymp \| D(\mathcal H_g (f) )\|_{H^q} \lesssim \|g'\|_{H^q} \|f\|_{H^p}
\end{align*}
which proves the wanted.\\
\\
 $(iii) \Longrightarrow (i)$ Since $g\in S^q\,\, (q>1)$ it holds that
$$
\|S_N(g') - g'\|_{H^q} \longrightarrow 0\,,\quad N\to \infty\,.
$$
For $N=2,3,...$ we define
$$
(\mathcal H_{g})_N(f)= \sum_{n=0}^N (n+1)\eta_{n+1} \sum_{k=0}^\infty \frac{a_k}{n+k+1} z^n\,,\quad z\in \mathbb D
$$
where $f(z)=\sum_{k=0}^\infty a_k z^k \in S^q\,.$

Set
\begin{equation*}
    \lambda_{n}=
\begin{cases}
    0 \quad\quad n=0,1,2,....,N-1\\
    \\
    (n+1)\sum_{k=0}^\infty \frac{a_k}{n+k+1} \,, \quad n=N,N+1,...
\end{cases}
\end{equation*}
then
\begin{align*}
    |\lambda_{n}|\leq (n+1)\sum_{k=0}^\infty \frac{|a_k|}{n+k+1}\leq \sum_{k=0}^\infty |a_k|\lesssim \|f\|_{S^p}
\end{align*}
and
\begin{align*}
    \sum_{j=2^n}^{2^{n+1}-1}&|\lambda_{j+1}-\lambda_j| =\sum_{j=2^n}^{2^{n+1}-1} \left|(j+2)\sum_{k-0}^\infty \frac{a_k}{j+k+2}-(j+1)\sum_{k=0}^\infty \frac{a_k}{j+k+1}\right|\\
    &\lesssim \sum_{j=2^n}^{2^{n+1}-1} \frac 1j \,\, \sum_{k=0}^\infty |a_k| \lesssim \|f\|_{S^p}\,.
\end{align*}
Thus,using Theorem \ref{LP-Zy}
\begin{align*}
    \|\mathcal H_g(f) - &(\mathcal H_g)_N (f)\|_{S^q}= \|\left(\mathcal H_g(f) - (\mathcal H_g)_N (f)\right)'\|_{H^q} \\
    & \simeq \|D\left(\mathcal H_g(f) - (\mathcal H_g)_N (f)\right)\|_{H^q}= \|D\left(\mathcal H_g(f) \right) - D\left((\mathcal H_g)_N (f)\right)\|_{H^q}\\
    &\quad \lesssim \|S_N(g') - g'\|_{H^q}\,\, \|f\|_{S^p}
    \end{align*}
which implies the compactness. \quad \quad \end{Pf}
\section{Final remarks}
\par Looking at Theorem \ref{summing up}, Theorem \ref{summing up2}, Theorem \ref{Rhaly Hardy} and Theorem \ref{Hg Hardy} it is natural to wander what happens with case $\mathcal R_{\eta}: H^p \to H^q,\,\, 2<p\leq q <\infty.$ So far, due to \cite{GGPSi},\cite{GG-2026JGA}, \cite{Bl-Ga-Gi2} we are aware of the succeeding conditions.
\begin{other}\label{Rhaly Hg qp}
    Let  $g(z)=\sum_{n=0}^\infty \eta_n z^n \in \hol(\mathbb D)$ and assume that $2<s< p < \infty$. \\
    \\
    $(i)$ If $g \in \Lambda^s_{\frac 1s} \,\,(\lambda^s_{\frac 1s})$ then
    $ \mathcal R_{\eta}\,,\,\,\mathcal H_g \in \mathcal B(H^p)\,\, (\mathcal K(H^p))$\\
    \\
    $(ii)$ If $ \mathcal R_{\eta}\,,\,\,\mathcal H_g \in \mathcal B(H^p)\,\, (\mathcal K(H^p))$ then $ g \in \Lambda^p_{\frac 1p}\,\, (\lambda^p_{\frac 1p})$\,.
\end{other}
It is interesting the subsequent observation about  $\mathcal R_{\eta} : H^p \to H^q$ which holds for any value of $p,q\in [1,\infty)$. A similar reformulation of the problem for the generalized Hilbert operator when acting as $\mathcal H_{g} : H^p \to H^p,\, p\in (1,\infty)$ can be found  in \cite{GP}.

If  $g(z)=\sum_{n=0}^\infty \eta_n z^n \in \hol(\mathbb D)$ and $f(z)=\sum_{k=0}^{\infty} a_k z^k \in H^p$ then
\begin{align*}
    \mathcal R_{\eta}(f) (z) &= \sum_{n=0}^{\infty} \eta_{n} \sum_{k=0}^n a_k z^n
    = \sum_{n=0}^{\infty} (n+1)\eta_{n} \( \frac{1}{n+1} \sum_{k=0}^n a_k \) z^n\\
    & = \sum_{n=0}^{\infty} (n+1)\eta_{n} \widehat{C(f)(n)} z^n\,,\,\,\,z\in \mathbb D
\end{align*}
where $\big\{\widehat{C(f)(n)}\big\}_{n\geq 0}=\big\{\frac{1}{n+1} \sum_{k=0}^n a_k \}_{n\geq 0}$ are the coefficients of the Ces\`{a}ro operator when acting on $H^{p}$. Then
\begin{align*}
    \mathcal R_{\eta}(f)(z) =   D(g) \star \mathcal C(f) (z)  \,,\,\,z\in \mathbb D
\end{align*}
where $D(g)(z)=\sum_{n=0}^{\infty} (n+1) \eta_{n} z^n \in \hol(\mathbb D)$.
\par As stated before, $\mathcal C \in \mathcal B(H^p)$ for any $p\in(0,\infty)$. Moreover,  the Ces\`{a}ro operator is injective \cite{CR}. Therefore the range $\mathcal C(H^p), 1\leq p<\infty,$ of the Ces\`{a}ro operator becomes a Banach space if it is employed with the norm
$$
\|\mathcal C(f)\|=\|f\|_{H^p}\,.
$$
We consider as
$
(\mathcal C(H^p), H^q)
$ the collection of those $h(z)=\sum_{n=0}^{\infty} h_n z^n \in \hol(\mathbb D)$ such that
$$
h\star \mathcal C(f) (z) = \sum_{n=0}^{\infty} h_n \( \frac{1}{n+1} \sum_{k=0}^n a_k \)  z^n \in H^q
$$
when $f\in H^p.$ In other words, $
(\mathcal C(H^p), H^q)$ consists of the coefficient(Hadamard) multipliers of the range $\mathcal C(H^p)$ of the Ces\`{a}ro operator to $H^q$.
\begin{proposition}\label{mult}
Let  $g(z)=\sum_{n=0}^\infty \eta_n z^n \in \hol(\mathbb D)$ and $1 \leq p, q<\infty$. The following are equivalent.\\
$(i)\,\, \mathcal R_{\eta} \in \mathcal B (H^p,H^q)$\\
$(ii)\,\,D(g) \in (\mathcal C(H^p), H^q) $
    \end{proposition}
\begin{proof}
    $(ii) \Rightarrow (i) : $ Assume that $D(g) \in (\mathcal C(H^p), H^q)$. Then
    \begin{align*}
        \|\mathcal R_{\eta} (f)\|_{H^q} &=\|D(g) \star \mathcal C(f)\|_{H^q}\\ &\lesssim \|\mathcal C(f)\|_{H^p} \lesssim \|f\|_{H^p}\,,\quad f \in H^p\,.
    \end{align*}
$(i) \Rightarrow (ii) : $  Since the Ces\`{a}ro is bounded and injective on $H^p$ for each $h \in \mathcal C(H^p) $ there is a unique $f\in H^p$ such that $h=\mathcal C(f)\,.$
Suppose that
$\mathcal R_{\eta} \in \mathcal B (H^p,H^q)$. Then
\begin{align*}
    \|D(g) \star h\|_{H^q}&=\|D(g) \star \mathcal C(f)\|_{H^q}
    \lesssim \|f\|_{H^p}\\
    &  = \|\mathcal C(f)\|=\|h\| \,,\quad h\in \mathcal C(H^p)\,.
\end{align*}
\end{proof}
According to the results presented  in the previous sections we are allowed to state that when $1<p \leq q \le 2$ or $1<p\leq 2 <q <\infty$
$$
D(g) \in (\mathcal C(H^p), H^q) \Leftrightarrow g \in \Lambda^q_{\frac 1p}\,.
$$
\par\medskip
{\bf Data Availability.} All data generated or analyzed during this
study are included in this article and in its bibliography

\par\medskip
{\bf Conflict of interest.} The authors declare that there is no
conflict of interest.

\par
\end{document}